\documentclass[a4paper,12pt,intlimits,oneside]{amsart}
\usepackage{amsmath}
\usepackage{amsthm}
\usepackage{latexsym}
\usepackage{amssymb}
\usepackage{xcolor}
\usepackage{graphicx}
\usepackage[colorlinks=true]{hyperref}
\numberwithin{figure}{section}
\def\R{{\mathbb R}}
\def\C{{\mathbb C}}

\def\T{{\mathbb T}}
\def\Z{{\mathbb Z}}

\def\N{{\mathbb N}}
\def\Q{{\mathbb Q}}

\def\e{\varepsilon}

\def\build#1_#2^#3{\mathrel{
\mathop{\kern 0pt#1}\limits_{#2}^{#3}}}
\def\td_#1,#2{\mathrel{\mathop{\build\longrightarrow_{#1\rightarrow #2}^{}}}}
\DeclareFontFamily{U}{MnSymbolC}{}
\DeclareSymbolFont{MnSyC}{U}{MnSymbolC}{m}{n}
\DeclareFontShape{U}{MnSymbolC}{m}{n}{
    <-6>  MnSymbolC5
   <6-7>  MnSymbolC6
   <7-8>  MnSymbolC7
   <8-9>  MnSymbolC8
   <9-10> MnSymbolC9
  <10-12> MnSymbolC10
  <12->   MnSymbolC12}{}
\DeclareMathSymbol{\intprod}{\mathbin}{MnSyC}{'270}
\newtheorem{theorem}{Theorem}

\newtheorem{proposition}{Proposition}
\newtheorem{lemma}{Lemma}
\newtheorem{remark}{Remark}
\newtheorem{definition}{Definition}
\begin{document}
\title[Periodic solutions of BO]{On the Benjamin--Ono equation on $\T$ and \\
its periodic and quasiperiodic solutions}

\author[P. G\'erard]{Patrick G\'erard}
\address{Laboratoire de Math\'ematiques d'Orsay, CNRS, Universit\'e Paris--Saclay, 91405 Orsay, 
France}\email{{\tt patrick.gerard@math.u-psud.fr}}
\author[T. Kappeler]{Thomas Kappeler}
\address{Institut f\"ur Mathematik, Universit\"at Z\"urich, Winterthurerstrasse 190, 8057 Zurich, Switzerland} 
\email{{\tt thomas.kappeler@math.uzh.ch}}
\author[P. Topalov]{Petar Topalov}
\address{Department of Mathematics, Northeastern University, Boston, MA 0215, USA}
\email{{\tt p.topalov@northeastern.edu}}



\begin{abstract}
In this paper, we survey our recent results on the Ben\-jamin-Ono equation on the torus. 
As an application of the methods developed we construct
large families of periodic or quasiperiodic solutions, which are not $C^\infty$-smooth.
\end{abstract}

\noindent

\thanks{
T.K. is partially supported by the Swiss National Science Foundation,
P.T. is partially supported by the Simons Foundation, Award \#526907.}

\maketitle



\section{Introduction}\label{introduction}
In this paper we consider the Benjamin-Ono (BO) equation on $\T$,
\begin{equation}\label{BO}
\partial_t u = H\partial^2_x u - \partial_x (u^2)\,, \qquad x \in \T:= \R/2\pi\Z\,, \,\,\, t \in \R,
\end{equation}
where $u\equiv u(t, x)$ is real valued  and $H$ denotes the Hilbert transform, defined as the Fourier multiplier
$$
f = \sum_{n \in \mathbb Z} \widehat f(n) e^{ i n x}  \mapsto
 \sum_{n \neq 0} -i \ \text{\rm sign}(n) \widehat f(n) \ e^{inx} \, .
$$
We refer to the recent survey \cite{Sa} for a discussion of the origin of this equation as a model for long, unidirectional
internal gravity waves in a two layer fluid and for a comprehensive bibliography.

Our study of the equation  \eqref{BO} focuses on the following topics: wellposedness in Sobolev spaces,
traveling waves and their orbital stability, long time behaviour of solutions (i.e., on properties of their orbits such as boundedness, 
orbital stability, recurrence),  aspects of integrability, and the construction of periodic and quasiperiodic solutions.

\medskip
\noindent
{\bf Sharp wellposedness in Sobolev spaces.} The wellposedness of \eqref{BO} in Sobolev spaces $H^s_r \equiv H^s(\T, \R)$
has been extensively studied in the last 40 years -- see e.g. references in \cite{MP}. To state our results we need to recall 
the following result of \cite{ABFS} (cf. also \cite{Sa0}). It says that for any initial data $u$ in $H^s_r$ with $s > 3/2$,
there exists a unique solution $v$ in $C(\R, H^s_r)$ of \eqref{BO} with $v|_{t=0} = u$ such that the solution map
$\mathcal{S}(t)u:=v(t)$,
\[
\mathcal S : \R \times H^s_r \to H^s_r , \, (t, u) \mapsto \mathcal S(t, u)\equiv\mathcal S(t)u,
\]
is continuous.
\begin{theorem}\label{wellposedness}
(i) For any $s > -1/2$, $\mathcal S$ extends as a continuous map, 
$\mathcal S : \R \times H^s_r \to H^s_r$ (\cite{GKT1}).\\
(ii) No such extension exists for $s=-1/2$. In more detail, there exists a sequence $(u^{(k)})_{k \ge 1}$
in $C^\infty(\T, \R)$ with the property that $u^{(k)} \to 0$ in $H^{-1/2}_r$, but the function
$
t \mapsto \langle \mathcal S(t)u^{(k)} \, | \, e^{ix} \rangle 
$
does not converge to $0$ on any time interval $I \subseteq \R$ with $|I| > 0$ (\cite{GKT1}).
\end{theorem}
Note that $\int_0^{2\pi} u dx$ is a prime integral of \eqref{BO}.
In particular, for any $a \in \R$ and $s > -1/2$, the solution map $\mathcal S(t)$ leaves the subspace
\[
H^s_{r, a} := \big\{ u \in H^s_r \,\big|\, \int_0^{2\pi} u\,dx = a\big\} 
\]
invariant.
We denote by  $\mathcal S_a(t)$ the restriction of $\mathcal S(t)$ to $H^s_{r,a}$.

\smallskip

\noindent
{\bf Addendum to Theorem \ref{wellposedness}.} {\em In \cite{GKT3}, \cite{GKT4} we prove that 
for any $t \in \R$, $a \in \R$, and  $-1/2 < s < 0$, $\mathcal S_a(t) : H^s_{r,a} \to H^s_{r,a}$ is nowhere locally 
uniformly continuous and that for any $s \ge 0$, $\mathcal S_a(t) : H^s_{r, a} \to H^s_{r, a}$ is real analytic.}

\begin{remark}
Theorem \ref{wellposedness}(i) improves on a result by Molinet \cite{Mol} (cf. also \cite{MP}),
saying that for any $s \ge 0$, $\mathcal S : \R \times H^s_r \to H^s_r$ extends as a continuous
map and Theorem \ref{wellposedness}(ii) improves on a result by Angulo Pava $\&$ Hakkaev \cite{AH}, saying
that no such extension exists for $s < -1/2$. 
The Addendum to Theorem \ref{wellposedness} improves on a result by Molinet \cite{Mol}--\cite{Mol0}, saying that for any 
$a \in \R$ and $s \ge 0$, $\mathcal S_a(t) : H^s_{r,a} \to H^s_{r,a}$ is analytic near the zero solution.

We refer to \cite{MP} for a comprehensive bibliography on the wellposedness of \eqref{BO}. 
We note that our method of proof is different from the methods used in the papers cited above
(cf. Theorem \ref{nonlinear FT} below). 
\end{remark}

\smallskip
\noindent
{\bf Traveling waves.} Recall that a solution $t \mapsto \mathcal S(t)U$ of \eqref{BO} with $U\in H^s_r$, $s > -1/2$,
is said to be a traveling wave with profile $U$ and velocity $c\in \R$ if $ \mathcal S(t)U = U(\cdot - ct)$.
Amick $\&$ Toland \cite{AT} listed all $C^\infty$--smooth profiles of traveling waves of \eqref{BO},
$$
U_{r, N, \alpha, a}(x) = N U_r(Nx + \alpha) + a\, ,\qquad N\in\N,\,\alpha \in \T, \, a \in \R\, ,
$$
where $U_r(x) = \frac{1-r^2}{1 - 2r\cos x + r^2}$, $0 < r < 1$, are the traveling wave profiles
with corresponding velocity $c_r := \frac{1+r^2}{1-r^2}$, found by Benjamin \cite{Benj}.

\begin{theorem}\label{traveling waves} For any $s > -1/2$ the following holds:\\
(i) Any traveling wave in $H^s_r$ has a profile of the form $U_{r, N, \alpha, a}$ and hence in particular
is $C^\infty$--smooth (\cite{GKT1}). \\
(ii) Any traveling wave is orbitally stable in $H^s_r$ (\cite{GKT1}, Remark \ref{tori are Liouville stable} below).
\end{theorem}

\begin{remark}
Theorem \ref{traveling waves}(ii) improves on a result by Angulo Pava $\&$ Natali \cite{AN}, saying that the traveling
waves with profile of the form $U_{r, N, \alpha, a}$ are orbitally stable in $H^{1/2}_r$.
\end{remark}

\noindent
{\bf Long time behaviour of solutions.} 
For solutions of the Benjamin-Ono equation on the line, the question of main interest
concerning their long time behaviour is to know whether they admit an asymptotic description as $t \to \infty$.
Since $\T$ is compact, such a description typically does not exist for solutions of \eqref{BO}.
In such a case, one is interested to know properties of the orbits of solutions such as boundedness, 
orbital stability, or recurrence.

\begin{theorem}\label{long time}
For any $u \in H^s_r$ with $s>  - 1/2$ the following holds:\\
(i) The orbit $\{ \mathcal S(t)u \,|\, t \in \R \}$ is relatively compact in $H^{s}_r$ (\cite{GKT1}).\\
(ii) The solution $\R \to H^s_r, \, t \mapsto \mathcal S(t)u$, is almost periodic (\cite{GKT1}).\\
(iii) $\sup_{t \in \R} \|  \mathcal S(t)u \|_s \le M$ where $M > 0$ can be chosen 
uniformly on bounded subsets of $H^s_r$ (\cite{GKT1}).
\end{theorem}

\noindent
{\bf Addendum to Theorem \ref{long time}.} {\em The solutions of Theorem \ref{wellposedness} of \eqref{BO} are orbitally stable 
in the sense explained in Remark \ref{tori are Liouville stable} below. }
\begin{remark}\label{remark 3}
We point out that a solution $\R \to H^s_r, \, t \mapsto \mathcal S(t)u$ of \eqref{BO} being almost periodic implies that it
is Poincar\'e recurrent. In particular, Theorem \ref{long time}(ii) improves on results by Deng $\&$ Tzvetkov $\&$ Visciglia \cite{DTV}
and Deng \cite{Deng}. In these papers (cf. also references in \cite{DTV}, \cite{Deng})
 invariant measures are constructed on Sobolev spaces of various order of regularity, which then are used to show that 
for a.e. initial data, the corresponding solutions are Poincar\'e recurrent.\\
Theorem \ref{long time}(iii) improves on results of similar type, which can be derived from the BO hierarchy, 
obtained in \cite{BK}, \cite{Nak}.  The BO hierarchy consists of a sequence $\mathcal{H}_j(u)$, $j \ge 0$, 
of prime integrals of \eqref{BO}.
The boundedness of $(\mathcal{H}_j)_{0 \le j \le n}$ can be shown to be equivalent to the boundedness of the $H^{n/2}$--norm.
Furthermore, Talbut \cite{Tal} proved such estimates for the $H^s$--norms, $-1/2 < s < 0$,
for smooth solutions of \eqref{BO}.
\end{remark}

\noindent
{\bf Nonlinear Fourier transform.} Our proofs of Theorem \ref{wellposedness} - Theorem \ref{long time}
rely on the integrability of the BO equation. In fact, we show that this equation is integrable in the strongest possible sense.
To state this result, we first need to introduce some more notation. 
As already mentioned above, $\int_0^{2\pi} u dx$ is a prime integral of \eqref{BO}.
Furthermore, for any solution $u(t, x)$ of \eqref{BO} and any $a \in \R$, 
$
u_a(t, x) = a + u(t, x - 2at)
$
is again a solution. We therefore restrict ourselves to consider equation \eqref{BO} on the Sobolev spaces $H^s_{r,0}$.
By $h^\sigma_+\equiv h^\sigma(\N, \C)$, $\sigma \in \R$, we denote the weighted $\ell^2$--sequence spaces defined by 
$h^\sigma_+:= \{ z = (z_n)_{n \ge 1} \,|\, z_n \in \C; \|z\|_\sigma < \infty \}$, where
$$
 \|z\|_\sigma := \big( \sum_{n=1}^\infty n^{2 \sigma} |z_n|^2 \big)^{1/2} \, .
$$
\begin{theorem}\label{nonlinear FT}(\cite{GK}, \cite{GKT1})
There exists a map
$$
\Phi : \bigsqcup_{s > -1/2} H^s_{r,0} \to  \bigsqcup_{s > -1/2} h_+^{s+1/2}, \ u \mapsto \zeta(u):= (\zeta_n(u))_{n \ge 1}
$$
so that the following properties hold for any $s > -\frac12$:\\
(NF1) $\Phi : H^s_{r,0} \to h_+^{s+1/2}$ is a homeomorphism and $\Phi$ and its inverse map bounded
subsets to bounded ones.\\
(NF2) For any $u \in H^s_{r,0}$, and any $n \ge 1$, $\zeta_n(\mathcal S(t)u) = e^{i\omega_n t} \zeta_n(u)$ where
\begin{equation}\label{formula frequencies intro}
\omega_n \equiv \omega_n(u) := n^2 - 2\sum_{k=1}^n k |\zeta_k(u)|^2 - 2n \sum_{k > n}  |\zeta_k(u)|^2\, .
\end{equation}
It follows that for any $n \ge 1$, $|\zeta_n(\mathcal S(t)u)|^2$ is independent of $t$. \\
(NF3) The map $\Phi$ does not continuously extend to a map $H^{-1/2} \to h^0_+$.
\end{theorem}

\noindent
{\bf Addendum to Theorem \ref{nonlinear FT}.} {\em In \cite{GKT3}, \cite{GKT4}, we prove that for any $s > -1/2$,
$\Phi : H^s_{r, 0} \to h_+^{s+1/2}$ and 
$\Phi^{-1} : h_+^{s+1/2} \to H^s_{r, 0}$ are real analytic.}

\begin{remark}\label{Birkhoff coordinates}
(i) The differential $d_0\Phi$ of $\Phi$ at $0$ is given by the weighted Fourier transform,
$\mathcal F[v] = - \big(\frac{1}{\sqrt{n}}\widehat v(n)\big)_{n \ge 1} $. Furthermore, the linearization of \eqref{BO} at the zero solution
is given by $\partial_t v = H\partial^2_x v$. The solution of the latter equation in $H^s_{r,0}$ are given by
$
\sum_{n\ne 0} e^{i\,\text{\rm sign}(n)n^2t} \widehat v(n) e^{inx} \, .
$
For this reason, we refer to $\Phi$ as a nonlinear Fourier transform.\\
(ii) It is well known that \eqref{BO} is Hamiltonian,
$$
\partial_t u = \partial_x \nabla \mathcal H\, , \qquad
\mathcal H(u) := \frac{1}{2\pi} \int_0^{2\pi} \Big(\frac{1}{2}\big(|\partial_x|^{1/2} u\big)^2 - \frac{1}{3} u^3\Big)\,dx,
$$
where $\partial_x$ is the Poisson structure, which corresponds to the Poisson bracket,
defined for functionals $F,$ $G$ on $H^s_{r,0}$ with sufficiently regular $L^2$--gradients,
$$
\{ F , \, G \} (u) = \frac{1}{2\pi} \int_0^{2\pi} (\partial_x \nabla F) \nabla G dx\, .
$$
We prove that for any $n$, $m \ge 1$,
$$
\{ \zeta_n , \zeta_m \} = 0\, , \qquad  \{ \zeta_n , \overline \zeta_m \} = -i \delta_{nm}\, ,
$$
implying that $\{ |\zeta_n|^2 , |\zeta_m|^2 \} = 0$. In addition, we show that $\mathcal H \circ \Phi^{-1}$
is a function of $|\zeta_n|^2$, $n \ge 1$, alone. Hence $\Phi$ is {\em canonical}, $|\zeta_n|^2$, $n \ge 1$, are
actions, and the phases of $\zeta_n$, $n \ge 1$, angles. In this way, the quantities $\zeta_n$, $n \ge1$, are globally defined 
Birkhoff coordinates of \eqref{BO} on $H^s_{r,0}$ for any $s > -1/2$.\\
(iii) For any $n \ge 1$, $\omega_n$ is referred to as the $n$th BO frequency. By \eqref{formula frequencies intro},
it is an {\em affine} function of the actions.
\end{remark}

\begin{remark}\label{tori are Liouville stable}
By Theorem \ref{nonlinear FT} one infers that for $\xi \in h_+^{s+1/2}$, $s > -1/2$, 
$$
\text{Iso}(\xi) :=\big\{ u \in H^s_{r, 0} \,\big|\, |\zeta_n(u)|^2 = |\xi_n|^2 \ \forall \, n \ge 1\big\} \, , 
$$ 
is an invariant torus for \eqref{BO}. Any such torus is {\em Lyapunov stable} in the sense that for any initial data $u \in H^s_{r, 0}$
near $\text{Iso}(\xi)$, the solution $\mathcal S(t)u$ stays close to $\text{Iso}(\xi)$ for all $t \in \R$.
\end{remark}

In the remaining part of this introduction, we briefly comment on applications of Theorem \ref{nonlinear FT}
and on elements of its proof. We keep our exposition as short as possible and refer to our papers for more details.

The Benjamin-Ono equation admits finite dimensional integrable subsystems. To define them, we need to introduce some
more notation. We say that $u \in \bigcup_{s > -1/2}H^s_{r,0}$ is a {\em finite gap potential} if there exists $N \in \N$ so that
$\zeta_n(u) = 0$ for any $n > N$. We denote by $\mathcal U_N$ the set of all such potentials in $\bigcup_{s > -1/2}H^s_{r,0}$ 
with $\zeta_N \ne 0$. Furthermore, we say that $u \in \bigcup_{s > -1/2}H^s_{r,0}$ is a {\em one gap potential} if there exists 
$N \ge 1$ so that $\zeta_n(u) \ne 0$ if and only if $n = N$. In particular such a potential
is in $\mathcal U_N$. Theorem \ref{nonlinear FT} implies that for any $N \ge 1$, $\mathcal U_N$ is contained 
$\bigcap_{s > -1/2}H^s_{r,0}$. An element $u \in \mathcal U_N$ is of the form 
$$
u(x) = - 2\mbox{Re} \Big( e^{i x} \frac{Q_N'(e^{i x})}{Q_N(e^{i x})} \Big) \, , \qquad
Q_N(z) = \prod_{j=1}^N(1 - q_j z)\, , 
$$
where $0 < |q_j| < 1$ for any $1 \le j \le N$ (\cite{GK}). The time evolution of potentials in $\mathcal U_N$
can be explicitly described, using the frequencies, defined in (NF2) of Theorem \ref{nonlinear FT}.
These solutions coincide with the ones constructed by Satsuma $\&$ Ishimori \cite{SI} and further
studied by Dobrokhotov $\&$ Krichever \cite{DK}. We refer to these solutions as finite gap solutions or
(periodic in $x$) multi--solitons. They are quasiperiodic in time. The one gap solutions coincide with the
traveling waves of Theorem \ref{traveling waves} and are periodic in time.

In Section \ref{application} we address the questions whether there are periodic and quasiperiodic solutions 
in time of \eqref{BO}  which are {\em not} (multi--)solitons.  
Both questions are answered affirmatively -- see Theorem \ref{periodic solutions which are not finite gap},
Proposition \ref{periodic finite gap solutions},
and Theorem \ref{nonsmooth qp solutions}.  The proof of these results is based on the action to frequency map, 
studied in Section \ref{action to frequency}. To the best of our knowledge, results of this type are not known 
for integrable PDEs such as the Korteweg-de Vries (KdV) equation or the nonlinear Schr\"odinger (NLS) equation.
We expect, but have not verified, that such results  
also hold for many of these PDEs although the action to frequency map
might be significantly more complicated  and hence the results more difficult to prove. 
In this connection, we only mention that the Hessian of the KdV and the NLS Hamiltonian are known to be strictly convex 
in a neighborhood of  the zero solution (cf. \cite{KMMT}, \cite{Molnar} for details).

\smallskip
A key ingredient of the proof of Theorem \ref{nonlinear FT} is the Lax pair formulation of \eqref{BO},
$\partial_t L_u = B_uL_u - L_uB_u $, where
\begin{equation}\label{Lax}
L_u = -i\partial_x - T_u\, , \qquad
B_u:= i (T_{|\partial_x|u} - T^2_u)\, ,
\end{equation} 
and $T_u$ denotes the Toeplitz operator, defined for potentials $u$ in $H^s_{r, 0}$, $s > -1/2$.
Here, the pseudo-differential operators $L_u$ and $B_u$ act on the Hardy space 
$$
H_+:=\big\{ u \in H^0(\T,\C) \,\big|\,\widehat u(n) = 0 \ \forall \, n < 0\big\}
$$
with $L_u$ being self-adjoint (cf. Corollary 2 in \cite{GKT1}).
The Lax pair formulation implies that the spectrum of $L_u$ is preserved by \eqref{BO}.
For any $u \in H^s_{r,0}$, the latter is discrete, bounded from below and consists of a sequence of simple real eigenvalues, which we
list in increasing order, $\lambda_0 < \lambda_1 < \cdots$  (cf. \cite{GK} ($s = 0$) and \cite{GKT1} ($-1/2 < s < 0$)).
They satisfy
\begin{equation}\label{def gaps}
\gamma_n:= \lambda_n - \lambda_{n-1} - 1 \ge 0 \, , \qquad \forall \ n \ge 1\, ,
\end{equation}
where $\gamma_n$, referred to as the $n$th gap of the spectrum of $L_u$ (cf. \cite[Appendix C]{GK}),
turns out to be the action action variable $|\zeta_n|^2$, mentioned in Remark \ref{Birkhoff coordinates}.
The spectrum of $L_u$ is encoded by the generating function,
\begin{equation}\label{def generating function}
\mathcal H_\lambda(u) := \langle (L_u + \lambda)^{-1} 1 | \, 1 \rangle .
\end{equation}
This function is at the heart of the construction of the map $\Phi$. 
It admits an expansion at $\lambda = \infty$, whose coefficients constitute the BO-hierarchy,
mentioned in Remark \ref{remark 3}. In Appendix \ref{generating function} we show that
$\mathcal H_\lambda(u)$ can be viewed as the relative determinant of $L_u + \lambda + 1$
with respect to $L_u + \lambda$.

\medskip
\noindent
{\em Acknowledgements.} We are grateful to E. Fouvry for drawing our attention to the reference \cite{W}.

\section{Action to frequency map}\label{action to frequency}
The aim of this section is to study the action to frequency map, mentioned in Section \ref{introduction}.
We restrict ourselves to potentials $u \in H^0_{r, 0}$.

Recall that for any $n \ge 1$, the actions $|\zeta_n(u)|^2$ of \eqref{BO},
associated to a potential $u \in L^2_{r,0} \equiv H^0_{r,0} $, coincide with the gap lengths $\gamma_n \equiv \gamma_n(u)$, 
$n \ge 1$, defined in \eqref{def gaps}.
By Theorem \ref{nonlinear FT}, the actions $\gamma(u):= (\gamma_n(u))_{n \ge 1}$ 
fill out the positive quadrant $\ell_{\ge 0}^{1,1}$ of the $\ell^1$-sequence space $\ell^{1,1} \equiv \ell^{1,1}(\N, \R)$,
$$
\ell_{\ge 0}^{1,1} :=\big\{ (x_n)_{n \ge 1} \in \ell^1(\N, \R) \,\big|\, 
x_n \ge 0 \  \forall \, n \ge 1\, , \sum_{n \ge 1} n x_n < \infty\big\} \, . 
$$
The frequencies $\omega_n \equiv \omega_n(\gamma)$, $n \ge 1$, 
(cf. \eqref{formula frequencies intro}), when viewed as
functions on the space of actions $\ell^{1,1}_{\ge 0}$,
can then be conveniently written as
\begin{equation}\label{formula frequencies}
\omega_n(\gamma) = n^2 - 2\check \omega_n(\gamma)\, , \qquad 
\check \omega_n(\gamma):=  \sum_{k\ge 1} k \gamma_k  - \sum_{k > n} (k-n) \gamma_k \, .
\end{equation}
Note that $\check \omega_n \equiv \check \omega_n(\gamma)$, $n \ge 1$, satisfy
\begin{equation}\label{property 1}
\lim_{n \to \infty} \check \omega_n = \sum_{k \ge 1} k \gamma_k\, , \qquad 
\check \omega_n - \check \omega_{n-1} = \sum_{k \ge n} \gamma_k \ge 0 \, , \ \  \forall n \ge 1 \, ,
\end{equation}
(with $ \check \omega_0(\gamma):= 0$) and
\begin{equation}\label{property 2}
(\check \omega_n - \check \omega_{n-1}) - (\check \omega_{n+1} - \check \omega_{n}) = \gamma_n \ge 0 \,, \qquad \forall \, n \ge 1\, .
\end{equation}
To describe the range of the map $\gamma \mapsto (\check \omega_n(\gamma))_{n \ge 1}$, we introduce the Banach space 
$\mathfrak c$, defined as the $\R$-vector space of real valued, convergent sequences ${\rm y} := (y_n)_{n \ge 1}$, endowed with 
the $\sup$-norm $\|{\rm y}\| := \sup_{n \ge 1} |y_n|$. For the sequel, it is convenient to set for any ${\rm y} = (y_n)_{n \ge 1}$ in 
$\mathfrak c$,
$$
y_0 := 0\, , \qquad y_\infty := \lim_{n \to \infty} y_n \, .
$$
Denote by $\mathfrak c_\uparrow$ the subset of $\mathfrak c$ of sequences ${\rm y} = (y_n)_{n \ge 1}$ satisfying
$$
0 \le  y_n \le y_{ n+1}\, , \quad (y_n - y_{n-1}) - (y_{n+1} - y_n) \ge 0\, , \qquad \forall \, n \ge 1\, .
$$
By \eqref{property 1}--\eqref{property 2}, for any $\gamma \in \ell^{1,1}_{\ge 0}$, 
the sequence $(\check \omega_n(\gamma))_{n \ge 1}$  is in $\mathfrak c_\uparrow$.
The {\em normalized action to frequency map} is defined as
$$
\check \Omega : \ell^{1,1}_{\ge 0} \to \mathfrak c_\uparrow, \, \gamma \mapsto (\check \omega_n(\gamma))_{n \ge 1} \, .  
$$
\begin{proposition}\label{check Omega homeo}
The map $\check \Omega : \ell^{1,1}_{\ge 0} \to \mathfrak c_\uparrow$ is a homeomorphism.
\end{proposition}
 
\begin{remark}
To the best of our knowledge, comparable results for integrable PDEs such as the KdV equation or the NLS equation are not known.
For partial results in this direction for the KdV equation, we refer to \cite{KMMT} and references therein.
\end{remark}

\begin{proof}
By \eqref{property 2}, $\check \Omega$ is one-to-one. To see that $\check \Omega$ is onto, consider ${\rm y} = (y_n)_{n \ge 1}$ in 
$\mathfrak c_\uparrow$.
Let
\begin{equation}\label{formula gamma_n}
\gamma_n: = (y_n - y_{n-1}) - (y_{n+1} - y_{n})\, , \qquad \forall \, n \ge 1 \, .
\end{equation}
Then by the definition of $\mathfrak c_\uparrow$, $\gamma_n \ge 0$ for any $n \ge 1$. By telescoping, one has for any $ n \ge 1$,
$$
y_n - y_{n-1} =
\lim_{N \to \infty} \sum_{k=n}^N  \gamma_k = \sum_{k \ge n} \gamma_k \, .
$$
and for any $n \ge 0$,
$$
y_\infty - y_n = \sum_{k \ge n} (y_{k+1} - y_k) = \sum_{k \ge n} \sum_{j > k} \gamma_j 
= \sum_{j > n} \gamma_j (j - n) \, ,
$$
or $y_n  = y_{\infty} -  \sum_{j > n} \gamma_j (j - n)$.
In particular, for $n=0$,
$ y_\infty = \sum_{k \ge 1} \ k \gamma_k$.
Hence $\gamma:= (\gamma_n)_{n \ge 1} \in \ell^{1,1}_{\ge 0}$ and 
$y_n = \sum_{k \ge 1} k \gamma_k - \sum_{k > n} (k - n) \gamma_k$ for any $n \ge 1$. 
We thus have proved that  $\check \Omega(\gamma) = {\rm y}$. 
 
Finally, note that $\check \Omega$  is the restriction of a linear map $\ell^{1,1} \to \mathfrak c$, whose norm is bounded by $2$. 
Indeed, since for any $(x_n)_{n \ge 1} \in \ell^{1,1}$ and any $n \ge 1$,
$$
\sum_{k\ge 1} k x_k  - \sum_{k > n} (k-n) x_k = \sum_{k =1}^n k x_k + n\sum_{k > n} x_k\, ,
$$
one has
\begin{equation}\label{eq:Om-bounded}
\big|\sum_{k\ge 1} k x_k  - \sum_{k > n} (k-n) x_k\big| \le \sum_{k =1}^n k |x_k| + n\sum_{k > n} |x_k|
\le 2 \sum_{k \ge 1} k |x_k|  \, .
\end{equation}
This implies that $\check \Omega: \ell^{1,1}_{\ge 0} \to \mathfrak c_\uparrow$ is continuous.
Going through the proof of the ontoness of 
$\check \Omega: \ell^{1,1}_{\ge 0} \to \mathfrak c_\uparrow$, one sees that 
$\check \Omega^{-1}: \mathfrak c_\uparrow \to \ell^{1,1}_{\ge 0} $
is continuous as well. Indeed, assume that $({\rm y}^{(k)})_{k \ge 1}$
is a sequence in $\mathfrak c_\uparrow$, converging
to ${\rm y} = (y_n)_{n \ge 1} \in \mathfrak c_\uparrow$. 
By \eqref{formula gamma_n} it follows that
for any $n \ge 1$, the nth component $\gamma_n^{(k)}$ of 
$\gamma^{(k)} := \check \Omega^{-1}({\rm y}^{(k)})$ converges 
to the nth component $\gamma_n$ of 
$\gamma :=\check \Omega^{-1}({\rm y})$ and that
$\sum_{n \ge 1} n \gamma_n^{(k)} = y_\infty^{(k)}$ converges
to  $\sum_{n \ge 1} n \gamma_n = y_\infty$. One then infers that
for any $N \ge 1$, $\sum_{n \ge N} n \gamma_n^{(k)}$ converges
to  $\sum_{n \ge N} n \gamma_n$. Given any $\e >0$, choose 
$N \ge 1$ so that $\sum_{n \ge N} n \gamma_n < \e/4$
and $k_\e \ge 1$ so that for any $k \ge k_\e$,
$\big|\sum_{n \ge N} n \gamma_n^{(k)} - \sum_{n \ge N} n \gamma_n\big| < \e /4$.
It then follows that $\big|\sum_{n \ge N} n \gamma_n^{(k)}\big| < \e/2$ for any $k \ge k_\e$.
From
$$
\sum_{n \ge 1} n  |\gamma_n^{(k)} - \gamma_n| \le 
\sum_{n \ge N} n \gamma_n^{(k)} + \sum_{n \ge N} n \gamma_n
 + \sum_{n < N} n  |\gamma_n^{(k)} - \gamma_n| ,
$$
one then concludes that $\gamma^{(k)}$ converges to $\gamma$ in $\ell^{1,1}$. 
\end{proof}
 
\begin{remark}\label{lacunary sequences of actions}
A result similar to the one of Proposition \ref{check Omega homeo} can be derived for
the restriction $\check \Omega_J$ of $\check \Omega $ to the subset $\ell^{1,1}_{\ge 0, J}$ of $\ell^{1,1}_{\ge 0}$, 
$$
\check \Omega_J : \ell^{1,1}_{\ge 0, J} \to \mathfrak c_{\uparrow, J} \, , 
\, (\gamma_{n_p})_{p} \mapsto (\check \omega_{n_p})_{p} \, .
$$
Here $J:= \{ n_1 < n_2 < \cdots < n_* \}$ is a subset of $\N$ with either $n_* = \infty$ or $n_* = N+1$ for some 
integer $N \ge 0$. The sets $\ell^{1,1}_{\ge 0, J}$ and $ \mathfrak c_{\uparrow, J}$ are defined 
in these two cases as follows.\\
(i) {\em Case $J$ infinite.} In this case, the subset $J$ is of the form
$J := \{n_p \,|\, p \ge 1 \}$, $\ell^{1,1}_{\ge 0, J}$ is the subset
$$
\ell^{1,1}_{\ge 0, J} :=\big\{ (\gamma_n)_{n \ge 1} \in \ell_{\ge 0}^{1,1} \,\big| \, 
\gamma_n > 0 \ \forall \, n \in J, \gamma_n = 0 \ \forall \, n \notin J\big\}  \, ,
$$
and $\mathfrak c_{\uparrow, J}$ the set of strictly increasing sequences ${\rm y}_J:=  (y_{n_p})_{p \ge 1} $ 
of positive numbers, satisfying
$$
y_{\infty} := \lim_{p \to \infty} y_{n_p} < \infty \, , \qquad 
\frac{y_{n_p} - y_{n_{p-1}}}{n_p - n_{p-1}}
- \frac{y_{n_{p+1}} - y_{n_p}}{n_{p+1} - n_{p}} > 0\, , \quad \forall \, p \ge 1\, ,
$$
where we set $n_0 = 0$ and $y_0 = 0$.
Note that for any $(\gamma_n)_{n \ge 1}$ in $\ell^{1,1}_{\ge 0, J} $,
$$
\check \omega_{n_p} - \check \omega_{n_{p-1}} = (n_{p} - n_{p-1} ) \sum_{q \ge p} \gamma_{n_q} \, , 
\qquad \forall \, p \ge 1\, ,
$$
where we recall that $ \check \omega_{0} = 0$. For any $p \ge 1$, one then has
$$
\frac{\check \omega_{n_p} - \check \omega_{n_{p-1}}}{n_{p} - n_{p-1} }  - 
\frac{\check \omega_{n_{p+1}} - \check \omega_{n_{p}}}{n_{p+1} - n_p }  = \gamma_{n_p} \, .
$$
(ii) {\em Case $J$ finite.} In this case, the subset $J$ is of the form
$J := \{n_p \,|\, 1 \le  p \le N \}$, 
$\ell^{1,1}_{\ge 0, J}$ denotes the subset
$$
\ell^{1,1}_{\ge 0, J} := 
\big\{ (\gamma_n)_{n \ge 1} \in \ell_{\ge 0}^{1,1} \,\big| \,  
\gamma_n > 0 \ \forall \, n \in J; \, \gamma_n = 0 \ \forall \, n \notin J\big\} ,
$$
and $\mathfrak c_{\uparrow, J}$ the set of strictly increasing finite sequences ${\rm y}_J:=  (y_{n_p})_{1 \le p  \le N} $ 
of real numbers, satisfying
$$
\frac{y_{n_p} - y_{n_{p-1}}}{n_p - n_{p-1}}
- \frac{y_{n_{p+1}} - y_{n_p}}{n_{p+1} - n_{p}} \ge 0\, , \quad \forall \, 1 \le p \le N\, ,
$$
where we set $n_0 = 0$, $y_0 = 0$ and $n_{N+1} = n_N + 1$, $y_{n_{N+1}} = y_{n_N}$.
Note that for any $(\gamma_n)_{n \ge 1}$ in $\ell^{1,1}_{\ge 0, J} $,
$$
\check \omega_{n_p} - \check \omega_{n_{p-1}} = (n_{p} - n_{p-1} ) \sum_{q \ge p} \gamma_{n_q} \, , 
\qquad \forall \, 1 \le p \le N\, .
$$
In this case, $\check \omega_n = \check \omega_{n_N}$ for any $n \ge n_N$
and for any $1 \le p \le N$, 
$$
\frac{\check \omega_{n_p} - \check \omega_{n_{p-1}}}{n_{p} - n_{p-1} }  - 
\frac{\check \omega_{n_{p+1}} - \check \omega_{n_{p}}}{n_{p+1} - n_p }  = \gamma_{n_p} \, .
$$
\end{remark}
 
\smallskip
 
It is convenient to extend $\check\Omega$ to a linear map $\Omega : \ell^{1,1}\to\mathfrak{c}$. 
This extension, is given by
$$
\Omega[{\rm x}] = \Big( \sum_{k= 1}^n k x_k  + n \sum_{k > n}  x_k \Big)_{n \ge 1} \, , \qquad \forall \,  
{\rm x} = (x_n)_{n \ge 1} \in \ell^{1,1} \, .
$$
Then $\Omega$ is a bounded by \eqref{eq:Om-bounded}. 
Denote by $Q$ the quadratic form, induced by $\Omega$. For any ${\rm x} \in \ell^{1,1}$,
$Q({\rm x})$ is given by
$$
Q({\rm x}) = \langle {\rm x} | \Omega({\rm x})  \rangle 
= \sum_{n \ge 1} x_n \sum_{k= 1}^n k x_k  +   \sum_{n \ge 1} n x_n  \sum_{k > n}  x_k.
$$
Since
$\sum_{n \ge 1} x_n \sum_{k= 1}^n k x_k = \sum_{k \ge 1} k x_k  \sum_{n \ge k} x_n$,
$Q({\rm x})$ can be written as
\begin{equation}\label{quadratic form 1}
Q({\rm x}) =  \sum_{n \ge 1} n x_n^2  + 2 \sum_{n \ge 1} n x_n  \sum_{k > n}  x_k \, .
\end{equation}
As a quadratic form, $Q$ extends to $\ell^{1, 1/2}\equiv \ell^{1, 1/2}(\N, \R)$ and
$$
\begin{aligned}
|Q( {\rm x})|  & \le  \sum_{n \ge 1}\big(\sqrt{n} |x_n|\big)^2
+2\sum_{n \ge 1} \sqrt{n} |x_n|  \sum_{k > n} \sqrt{k} |x_k| \\
&  \le 2\sum_{n \ge 1} \sqrt{n} |x_n|  \sum_{k \ge n} \sqrt{k} |x_k|   \le 2  \| {\rm x} \|^2_{\ell^{1, 1/2}} \, .
\end{aligned}
$$
(We mention that the quadrant $\ell^{1, 1/2}_{\ge 0}$, filled out by the actions,  corresponds to the phase space of 
potentials $H^{-1/4}_{r, 0}$, for which the map $\Phi$ is well defined by Theorem \ref{nonlinear FT}.)
By \eqref{quadratic form 1} one has 
$Q( {\rm x}) =  \sum_{n \ge 1} n\big(x_n^2 + 2 x_n \sum_{k > n}  x_k\big)$.
Hence by completing squares one obtains 
$$
Q( {\rm x}) =  \sum_{n \ge 1} n\Big(x_n + \sum_{k > n}  x_k\Big)^2  -  \sum_{n \ge 1} n\Big( \sum_{k > n}  x_k\Big)^2
$$
or $Q( {\rm x}) =  \sum_{n \ge 1} n (s_n^2 - s_{n+1}^2 )$ where $s_n :=   \sum_{k \ge n}  x_k$. It implies
$$
Q( {\rm x}) = s_1^2 + \sum_{n \ge 2} s_n^2 =  \sum_{n \ge 1} s_n^2 \, .
$$
In particular, one sees that $Q$ is a positive semidefinite quadratic form on $\ell^{1, 1/2}$ and that
$-Q$, when restricted to $\ell^{1,1}_{\ge 0}$, coincides with the quadratic part the Hamiltonian $\mathcal H$ of the BO equation,
when expressed in the action variables $\gamma=(\gamma_n)_{n \ge 1}$ (cf. \cite[Proposition 8.1]{GK}),
\begin{equation}\label{Hamiltonian convex}
\mathcal H = \sum_{n \ge 1} n^2 \gamma_n - \sum_{n \ge 1}\Big(\sum_{k \ge n} \gamma_n\Big)^2 \, .
\end{equation}
Since $-2Q$ is the Hessian of  $\mathcal H$, the latter can thus be viewed as a concave function.
In summary, $Q$ has the following properties.
  
\begin{proposition}\label{results on Q}
The quadratic form $Q$ is well defined on $\ell^{1, 1/2}$. It is positive semidefinite
and satisfies
$$
|Q( {\rm x})| \le  2  \| {\rm x} \|^2_{\ell^{1, 1/2}} \, , \qquad \forall \, {\rm x} \in \ell^{1, 1/2} \, .
$$
Furthermore, 
\begin{equation}\label{lower bound Q}
\inf\big\{ Q({\rm x}) \,\big|\, \| {\rm x} \|_{\ell^{1, 1/2}} = 1\big\} = 0\, ,
\end{equation}
hence $Q$ is {\em not} positive definite. 
\end{proposition}

\begin{proof}
It remains to prove \eqref{lower bound Q}. For an arbitary integer $N \ge 2$, consider the sequence
${\rm x}_N := (x_n^{(N)})_{n \ge 1} \in \ell^{1, 1/2}$ with $x_n^{(N)} = 0$ for any $n > N$ and
$$
x_n^{(N)} = \frac{1}{a_N}(-1)^{n+1} \quad \forall \,  1 \le n \le N \, ,  \qquad  a_N := \sum_{n=1}^N \sqrt{n} \, .
$$
Then $\| {\rm x}_N\|_{\ell^{1, 1/2}} = 1$ and $(s^{(N)}_n)^2 \le 1/a_N^2$, implying that
$Q({\rm x}_N) \le \frac{N}{a_N^2}$. Since
$$
a_N \ge \int_0^N \sqrt{x} dx = \frac23 N^{3/2}, \qquad
a_N^2 \ge \frac49 N^3\, ,
$$
it follows that $ Q({\rm x}_N) \le \frac{9}{4N^2}$. As $N \ge 2$ is arbitrary, \eqref{lower bound Q} holds.
\end{proof}
 \medskip


\section{Applications}\label{application}

As an illustration of our analysis of the action to frequency map we apply our results to construct  families of periodic solutions
of the BO equation, which are {\em not} traveling waves, and families of quasiperiodic solutions, which are {\em not} finite gap solutions.

\smallskip
\noindent
{\bf Periodic solutions.} Our first result addresses the question, whether  there are periodic in time solutions of \eqref{BO}, which are {\em not
finite gap solutions}.
  
\begin{theorem}\label{periodic solutions which are not finite gap}
(i) For  $T > 0$ with $T/\pi $ rational, any $T$--periodic solution in $L^2_{r,0}$ of \eqref{BO} is a finite gap solution.\\
(ii) For any positive irrational number $b$, there exists a strictly increasing sequence $(n_p)_{p \ge 1}$ in $\N$ and a sequence of actions
$\gamma = (\gamma_n)_{n \ge1}$ in $\ell^{1,1}_{\ge 0, J}$, $J:= \{ n_p \,:\, p \in \N \}$, satisfying
$$
\sum_{p \ge 1} n_p^3 \gamma_{n_p} = \infty \, , \qquad  
\omega_{n_p}(\gamma) \in   b \Z \, , \quad \forall \, p \ge 1\,,
$$
where $\omega_n(\gamma)$ is given by \eqref{formula frequencies}.
As a consequence, any potential $u_0$ in the torus (cf. \cite[Section 3]{GK})
\begin{equation}\label{isospectral set}
\text{Iso}_\gamma :=\big\{ u \in L^2_{r, 0} \,\big| \,  \gamma_n(u) = \gamma_n \  \forall \, n \ge 1\big\}
\end{equation}
is {\em not} in $H^1_{r,0}$ and the solution $u(t)$ of the BO equation with $u(0)= u_0$
(cf. Theorem \ref{wellposedness}) is periodic in time with period $T= 2\pi / b$.
Therefore, $\text{Iso}_\gamma$ is entirely filled up with $T$-periodic solutions.
\end{theorem}
 
\begin{proof}
 (i)
  Let $u(t)$ be the solution of \eqref{BO} with initial condition $u(0) = u_0 \in L^2_{r, 0}$.
By Theorem \ref{wellposedness}, $\Phi (u(t))=(\zeta_n(t))_{n\ge 1} \in h_+^{1/2}$ is given by
$$
\zeta_n(t)=\zeta_n(0){\rm e}^{it\omega_n}\, , \qquad  |\zeta_n(0)|^2 = \gamma_n \equiv \gamma_n(u_0), \qquad \forall \, n \ge 1\, ,
$$
Hence $u$ is $T$--periodic if and only if, for every $n \ge 1$ with $\zeta_n(0)\ne 0$, 
$$
\omega_n \in \omega \Z\ ,\qquad \omega :=\frac{2\pi}{T}\ .
$$
By assumption, $\omega$ is rational.  Choose $p, q \in \Z$ with  $q \ge 1$ so that 
$\omega = p/q$. Since $u_0$ is in $L^2_{r,0}$, $\gamma = (\gamma_n)_{n \ge 1} \in \ell^{1,1}_{\ge 0}$
and formula \eqref{formula frequencies} for the frequencies hold.
It then follows  that for any $n \ge 1$ with $\zeta_n(0)\ne 0$, 
$$
-2 \sum_{k\ge 1} k \gamma_k + 2\sum_{k>n}(k-n) \gamma_k \in \frac{1}{q}\,Z \ .
$$
Assume that there are {\em infinitely} many integers $n$ with $\zeta_n(0)\ne 0$.
Since $\sum_{k>n}(k-n) \gamma_k \to 0$ as $n \to \infty$, one then concludes that
$-2 \sum_{k\ge 1} k \gamma_k \in \frac 1q \Z$. Consequently, for infinitely many integers $n$, one has
$2\sum_{k>n}(j-n)|\zeta_k(0)|^2\in \frac 1q \Z$, which
contradicts that $\sum_{k>n}(k-n) \gamma_k$ converges to $0$. Hence there are only finitely 
many integers $n$ with $\zeta_n(0)\ne 0$, which implies that $u(t)$ is a  finite gap solution. \\
(ii) Our task is to find a strictly increasing sequence $(n_p)_{p \ge 1}$ of $\N$ 
and a sequence $\gamma = (\gamma_n)_{n \ge 1}$ 
in $\ell^{1,1}_{\ge 0, J}$, 
$J:= \{ n_p \,|\, p \in \N \}$, (cf. Remark \ref{lacunary sequences of actions}(i)), so that there exists
a sequence $(m_{p})_{p \ge 1}$ in $\Z$ with the property that 
\begin{equation}\label{identity for frequencies}
n_p^2 -2 \check \omega_{n_p} = m_p b\, , \qquad \forall \, p \ge 1 \, ,
\end{equation}
and 
\begin{equation}\label{identity y infty}
\check \omega_{n_p} =  y_\infty - \sum_{q > p} (n_q - n_p) \gamma_{n_q} \, ,  
\qquad  y_\infty =  \sum_{p\ge 1} n_p \gamma_{n_p} \, .
\end{equation}
To find such sequences, we use that by a result due to Weyl \cite{W}, 
the set $D_b:= \{ n^2 + k b \,|\, n \in \Z_{\ge 0}, \, k \in \Z \}$ is dense in $\R$. 
Given an arbitrary positive real number $y_\infty$, choose a sequence $(\e_p)_{p \ge 1}$ of the form
\begin{equation}\label{def varepsilon}
\e_p:= \e_0 4^{-p}\, , \qquad \forall \,  p \ge 1,
\end{equation}
where $\e_0>0$ is chosen so that
\begin{equation}\label{bound epsilon_0}
\e_0 < 4 y_\infty\ .
\end{equation}
For any $p \ge 1$, we then choose integers $n_p \ge 0$  and $k_p$ so that
\begin{equation}\label{choice rho}
\rho_{p} := 2 y_\infty -  n^2_p  + k_p b  \in [\e_p, 2\e_p]\ .
\end{equation}
By the definition of the sequence $(\e_p)_{p \ge 1}$, $(\rho_{p})_{p \ge 1}$ is a strictly decreasing sequence of positive numbers, 
converging to $0$ as $n \to \infty$,
$$
0 < \rho_{p+1}\leq 2\e_{p+1}=\e_p/2< \e_p \le \rho_p \, .
$$
By induction on $p$, it is possible to choose $n_p \ge 1$ for any $p \ge 1$, so that $n_1 \ge 1$ and $n_{p+1}\ge 2n_p$. 
(Indeed, for every integer $N$, the set 
$$\{ n^2+ kb \,|\,   0 \le n \le N, \ k\in \Z \}$$
is discrete, so what is left over after removing it from $D_b$ is still dense in $\R$.)
Thinking of  $\rho_{p}$ as $2y_\infty-2\check\omega_{n_p}$ and hence of $2y_\infty -  \rho_{p}$  as $2 \check \omega_{n_p}$,
we define  for any $p\ge 2$ (cf. Remark \ref{lacunary sequences of actions}(i)),
\begin{equation}\label{gamma n_p}
\gamma_{n_p}:= a_p - a_{p+1}\, , \qquad
a_p:= \frac{\rho_{p-1}-\rho_p}{2(n_p-n_{p-1})} > 0 \, .
\end{equation}
Using that $n_p \ge n_p-n_{p-1}$ and $n_{p+1}-n_p\ge n_p$, one sees that
\begin{equation}\label{gamma n_p'}
\frac{n_p}{n_p-n_{p-1}}(\rho_{p-1}-\rho_p) - \frac{n_p}{n_{p+1}-n_{p}}(\rho_p - \rho_{p+1}) 
\ge \rho_{p-1} - 2\rho_p + \rho_{p+1} \, .
\end{equation}
Since by \eqref{choice rho}, $\e_p \le \rho_p \le 2 \e_p$ and  by \eqref{def varepsilon}, $\e_{p-1} - 4 \e_p = 0 $ ,  
one gets from \eqref{gamma n_p} and \eqref{gamma n_p'},
$$
2 n_p\gamma_{n_p} \ge   \e_{p-1}-4\e_p+\e_{p+1} = \e_{p+1} .
$$
Consequently, $\gamma_{n_p}>0$ and, for some $c > 0$,
$$
n_p^3 2\gamma_{n_p} \ge  (2^{p-1} n_1)^2 2 n_p \gamma_{n_p} \ge  4^{p-1} n_1^2  \e_{p+1} 
= 4^{p-1} n_1^2 \frac{\e_0}{4^{p+1}} \ge c>0 \, ,
$$
implying that  $\sum n_p^3\gamma_{n_p} = \infty$. 
On the other hand by \eqref{gamma n_p},
$$
2n_p\gamma_{n_p}\le \frac{n_p}{n_{p} - n_{p-1}}( \rho_{p-1}-\rho_p) \le  \rho_{p-1}-\rho_p  \, , 
$$
hence by telescoping and by the bound \eqref{bound epsilon_0} of $\e_0$,
$$\sum_{p\ge 2}n_p\gamma_{n_{p}}\le \rho_1/2 \leq \e_1=\e_0/4 < y_\infty\ .$$
Now define $\gamma_{n_1}>0$ so that the second identity in \eqref{identity y infty} holds,
$$
\gamma_{n_1} := \frac{1}{n_1}\Big( y_\infty - \sum_{p \ge 2}n_p \gamma_{n_p}\Big) > 0 \, .
$$ 
It remains to check the  identities in \eqref{identity for frequencies}. 
Using the definition \eqref{gamma n_p} of $\gamma_{n_p}$, $p \ge 2$,
one verifies that for any $p \ge 1$, 
$\rho_p =2\sum_{q>p}(n_q-n_p)\gamma_{n_q}$ and thus, by the definition of $\rho_p$,
$$
{\check\omega}_p=n_p^2- 2y_\infty+2\sum_{q>p}(n_q-n_p)\gamma_{n_q}= k_p b \, , \qquad  \forall \, p\ge 1\,  .
$$
This completes the proof of item (ii).
\end{proof}

\begin{remark}\label{lacunary structure}
(i) It is possible to choose the sequence $(\e_p)_{p \ge 1}$, constructed in the proof of Theorem \ref{periodic solutions which are not finite gap}(ii), 
so that $u \notin H^s_{r, 0}$ for some $0 < s<1$.\\
(ii) The sequence $(n_p)_{p \ge 1}$, constructed in the proof of Theorem \ref{periodic solutions which are not finite gap}(ii),
needs to be sparse in the following sense: if $(n_p)_{p \ge 1}$ is a strictly increasing sequence in $\N$
and $\gamma = (\gamma_k)_{k \ge 1}$ a sequence of actions in $\ell^{1,1}_{\ge 0, J}$, $J:= \{ n_p \,|\, p \ge 1\}$, 
with the property that there exists an infinite set of integers $n$ in $J$ so that $n-1$ and $n+1$  are also contained in $J$,
then the frequencies  cannot satisfy \eqref{identity for frequencies}.
Indeed, the frequencies satisfy on $\ell^{1,1}_{\ge 0}$ the identities (with $\omega_0 = 0$)
$$\omega_{k+1}-2\omega_k+\omega_{k-1}=2+2\gamma_k\, , \qquad \forall \, k \ge 1.$$
Hence if $\omega_{n_p} \in b \Z$ for any $p \ge1 $, it then would follow that $2+2\gamma_n(u)\in b \Z$
for infinitely  many $n$ in $J$, implying that $2 \in b\Z$. This however contradicts the assumption of $b$ being irrational.
\end{remark}

The following result says that that there are many finite gap solutions of the BO equation which are periodic in time,
but not traveling waves.

\begin{proposition}\label{periodic finite gap solutions}
For any rational number of the form $ 1/a$, $a \in \N$, any $N\in \N$ , and any strictly increasing sequences $(n_p)_{1 \le p \le N}$,
$(k_p)_{1 \le p \le N}$ in $\N$ with 
$$
\frac{k_p - k_{p-1}}{n_p - n_{p-1}} -  \frac{k_{p+1} - k_{p}}{n_{p+1}- n_{p}} > 0\, , \qquad \forall \, 1 \le p \le N\, ,
$$
(where we set $n_0 := 0$, $n_{N+1}:= n_N +1$, $k_0:= 0$, and $k_{N+1} := k_N$), the following holds: the
sequence of actions, 
$\gamma = (\gamma_n)_{n \ge 1} \in \ell^{1,1}_{\ge 0, J}$, defined by
$$
\gamma_{n_p} := \frac{1}{a} \Big( \frac{k_p - k_{p-1}}{n_p - n_{p-1}} -  \frac{k_{p+1} - k_{p}}{n_{p+1}- n_{p}}\Big) \, , 
\qquad  \forall \, 1 \le p \le N\, ,
$$
and $J:= \{ n_p \,|\, 1 \le p \le N \}$ (cf. Remark \ref{lacunary sequences of actions}(ii)),
has frequencies $\omega_{n_p} \equiv \omega_{n_p} (\gamma)$, $1 \le p \le N$, given by
$$
\omega_{n_p}  = n_p^2 -2 \check \omega_{n_p} \in \frac{1}{a}\,Z \, , \qquad
\check \omega_{n_p} = \frac{1}{a}\,k_p  \, , \qquad  \forall \, 1 \le p \le N\, .
$$
As a consequence, any potential $u_0$ in the torus $\text{Iso}_\gamma$ (cf. \eqref{isospectral set})
is a finite gap potential, the solution $u(t)$ of \eqref{BO} with $u(0) = u_0$ periodic in time with period 
$T= 2\pi a$, and hence $\text{Iso}_\gamma$ entirely filled with $T$-periodic solutions.
\end{proposition}
 
\begin{remark}
Since by Theorem \ref{traveling waves}, the traveling waves of the BO equation coincide with the one gap solutions, 
it follows from Proposition \ref{periodic finite gap solutions} that there is a plentitude of periodic in time solutions of \eqref{BO} 
which are finite gap solutions, but not traveling waves.
\end{remark}
 
\begin{proof}
The claimed results follow from Remark \ref{lacunary sequences of actions}(ii).
\end{proof}

\smallskip
\noindent
{\bf Quasiperiodic solutions.}
The aim of this paragraph to construct quasiperiodic solutions of \eqref{BO}, which are {\em not} finite gap solutions.
We begin with describing the  $\omega$--quasiperiodic in time solutions of \eqref{BO}
in terms of the map $\Phi$ of Theorem \ref{nonlinear FT}
where $\omega$ is a frequency vector in  $\R^d$, $d \ge 2$, with $\Q $--linearly independent components. 
\begin{definition}\label{quasip}
Let $E$ be a Banach space and $\omega \in \R^d$, $d \ge 2$ with $\Q $--linearly independent components. 
A function $u\in C(\R, E)$ is said to be $\omega$--quasiperiodic if there exists
a function $U\in C(\T^d,E)$,
so that $u(t)=U(t\omega)$ for any $t\in \R$. 
Here, by notational convenience,
the vector $t\omega$ denotes also the class of vectors $t\omega +(2\pi \Z)^d$. 
The function $U$ is referred to as the {\em profile} of $u$.
\end{definition} 

\begin{proposition}\label{quasipHs}
Let $U : \T^d \to H^s_{r,0}$ with $s>-1/2$ and let $\omega$ be a vector in $\R^d$, $d \ge 2$,
with $\Q$--linearly independent components. 
Then
$U$ is the profile of a $\omega $--quasiperiodic solution of \eqref{BO} in $H^s_{r,0}$ if and only if
$\Phi (U(\varphi ))$ is of the form
\begin{equation}\label{form of Phi U}
\Phi (U(\varphi ))=\big(\zeta_n{\rm e}^{i k^{(n)} \cdot \varphi } \big)_{n\ge 1}\,  , \qquad \forall \,  \varphi \in \T^d \, ,
\end{equation}
where $(\zeta_n)_{n\ge 1}\in h_+^{s+1/2}$ and $(k^{(n)})_{n\ge 1}$ is a sequence in $\Z^d$ with the property 
that for any $n \ge 1$,
\begin{equation}\label{dichotomy}
\zeta_n=0 \qquad {\rm or } \qquad   k^{(n)} \cdot \omega = n^2 - 2\sum_{k=1}^n k |\zeta_k|^2 - 2n \sum_{k > n}  |\zeta_k|^2 \,  .
\end{equation}
\end{proposition}

\begin{remark}
For any $\omega$-quasiperiodic solution of \eqref{BO} with action variables 
$\gamma = (\gamma_n)_{n \ge 1} \in \ell^{1, 1 + 2s}_{\ge 0}$, the invariant torus
$$
\text{Iso}_\gamma := \big\{ u \in H^s_{r, 0} \,\big| \, \gamma_n(u) = \gamma_n \ \forall \, n \ge 1\big\}
$$ 
is filled with $\omega$-quasiperiodic solutions of \eqref{BO}. The corresponding profiles
are given by \eqref{form of Phi U} with $(\zeta_n)_{n \ge 1}$ being an arbitrary element in the set $\Phi(\text{Iso}_\gamma)$.
\end{remark}

\begin{proof}
Let $U$ be the profile of an $\omega $--quasiperiodic solution $u(t)$  in $H^s_{r,0}(\T )$ of \eqref{BO}. 
It is to show that \eqref{form of Phi U}--\eqref{dichotomy} hold.
Let
$$ (\xi_n(\varphi ))_{n\ge 1}:= \Phi (U(\varphi )) \in h_+^{s+1/2}\, \qquad \forall \, \varphi \in \T^d .$$
Since by Definition \ref{quasip}, $U$ is in $C(\T^d,H^s_{r,0})$, the map
$\T^d \to h_+^{s+1/2}, \,  \varphi \mapsto (\xi_n(\varphi))_{n \ge 1}$ is continuous
and by (NF2) in Theorem \ref{wellposedness}, for any $n \ge 1$,
\begin{equation}\label{formula time evolution of xi}
\zeta_n{\rm e}^{it\omega_n}  = \xi_n(t\omega )\, , \qquad \,  \forall \, t \in \R \, ,
\end{equation}
where $(\zeta_n)_{n \ge 1} := (\xi_n(0))_{n \ge 1}\in h_+^{s+1/2}$.
Since by assumption, the components of $\omega $ are linearly independent in $\Q$, 
the Fourier coefficients $\widehat \xi_n (k)$, $k\in \Z^d$, of $\xi_n$ can be computed as
$$\widehat \xi_n (k) = (2\pi)^{-d}\int_{\T^d}\xi_n(\varphi ){\rm e}^{-i  k \cdot \varphi }\, d\varphi 
=\lim_{T\to \infty }T^{-1}\int_0^T \xi_n(t\omega ){\rm e}^{-it k \cdot \omega }\, dt \, .
$$
Furthermore, by the formula \eqref{formula time evolution of xi} for $\xi_n(t\omega )$, one has
$$
\lim_{T\to \infty }T^{-1}\int_0^T \xi_n(t\omega ){\rm e}^{-it k \cdot \omega }\, dt =
\zeta_n \lim_{T\to \infty }T^{-1}\int_0^T {\rm e}^{it (\omega_n - k \cdot \omega) }\, dt\, .
$$
Note that the right hand side of the latter identity vanishes if $ \omega_n \ne k \cdot \omega  $, and equals 
$\zeta_n$ if $\omega_n =  k \cdot \omega  $.
Consequently, for any given $n\ge 1$, the following dichotomy holds:  in the case where there is no $k\in \Z^d$, 
satisfying $\omega_n = k \cdot \omega $,  
it follows that $\widehat \xi_n (k)  = 0$
for any $k \in \Z^d$. Hence the continuous function $\xi_n$ vanishes, implying that $\zeta_n = \xi_n(0) = 0$. 
Otherwise, since the components of $\omega$ are linearly independent over $\Q$, 
there exists exactly one $k^{(n)}\in \Z^d$ such that $\omega_n = k^{(n)} \cdot \omega $
and $ \xi_n(\varphi )$ equals $\zeta_n {\rm e}^{i k^{(n)} \cdot \varphi  }$. We thus have proved that \eqref{form of Phi U}--\eqref{dichotomy} hold.
 
Conversely, if $\Phi \circ U$ is given by the expression \eqref{form of Phi U}, 
$U$ is a continuous map $\T^d \to H^s_{r,0}$ since $\Phi^{-1} $ is continuous. Furthermore by \eqref{dichotomy},
$\Phi(U(t\omega ))=(\zeta_n{\rm e}^{it\omega_n})_{n\ge 1}$
so that $t\mapsto U(t\omega )$ is a $\omega$-quasiperiodic solution of \eqref{BO} with profile $U$.
\end{proof}

The following result illustrates how Proposition \ref{quasipHs} can be used to construct
 $\omega$-quasiperiodic solutions of \eqref{BO},
which are not $C^\infty$-smooth, hence in particular not finite gap solutions.

\begin{theorem}\label{nonsmooth qp solutions}
Let $b $ be an irrational real number and $\omega :=(1, b)\in \R^2$. For any $s>-1/2$, there exists a $\omega$--quasiperiodic solution 
 of the BO equation in $H^s_{r,0} \setminus \bigcup_{\sigma>s} H^\sigma_{r, 0}$.
\end{theorem}

\begin{proof}
Let $s > -1/2$ be given. 
In view of Proposition \ref{quasipHs}, it suffices to find 
a sequence $(\zeta_n)_{n\ge 1}$ in $h_+^{s+1/2} \setminus \cup_{\sigma>s} h_+^{\sigma +1/2}$ with $\zeta_n = |\zeta_n| > 0$
and a sequence $(k^{(n)})_{n\ge 1}$ in $\Z^2$ so that 
$$  
k^{(n)} \cdot \omega  = n^2 - 2\sum_{k=1}^n k \zeta_k^2 - 2n \sum_{k > n}  \zeta_k^2\, , 
\qquad \, \forall n\ge 1 \, .
$$
The latter identities imply that for any $n \ge 1$,
\begin{eqnarray*}
( k^{(n+1)} - k^{(n)} )\cdot \omega  &=&2n+1-2\sum_{j=n+1}^\infty \zeta_j^2\ ,\\
(  k^{(n+1)}-2k^{(n)}+k^{(n-1)} )\cdot \omega &=&2+2\zeta_n^2
\end{eqnarray*}
with $k^{(0)}:= 0 \in \Z^2$. It is convenient to reformulate our problem. Let
$$
\ell^{(n)}:=  k^{(n+1)}-2k^{(n)}+k^{(n-1)}\, , \qquad \forall \, n \ge 1\, .
$$
Since $ k^{(n+1)}=  \ell^{(n)} + 2k^{(n)} - k^{(n-1)}$, $n\ge 1$,
our problem can be described equivalently as follows:  find a sequence $(\gamma_n)_{n\ge 1}$ in $\R_{>0}$,  
belonging to $ \ell_+^{1,1+2s}\setminus \cup_{\sigma>s}\ell_+^{1,1+2\sigma }$, a sequence $(\ell^{(n)})_{n\ge 1}$ in $\Z^2$, 
and  $k^{(1)}\in \Z^2$ so that $k^{(0)}=0$ and
\begin{equation}\label{setup problem}
 \ell^{(n)} \cdot \omega  =2+2\gamma_n\, ,  \quad \forall  \, n\ge 1\, , \qquad   k^{(1)} \cdot \omega  =1-2\sum_{j=1}^\infty \gamma_j  \,  ,
\end{equation}
where for any $n \ge 1$, $\gamma_n$ is related to $\zeta_n$ by $\zeta_n=\sqrt{\gamma_n}$. 

Using the density of the additive subgroup $ \omega \cdot \Z^2   =\Z + b \Z$ in $\R $, 
it is straightforward to construct sequences $(\gamma_n)_{n\ge 1}$ and $(\ell^{(n)})_{n\ge 1}$,
which satisfy the first set of identities in \eqref{setup problem}. 
But it is more involved to construct such sequences satisfying at the same time
the second identity in \eqref{setup problem}, which can be rephrased as
\begin{equation}\label{condition k^1}
1-2\sum_{j=1}^\infty \gamma_j\in \omega \cdot \Z^2\, .
\end{equation}
Accordingly, we proceed in two steps: \\
{\em Step 1.}
Let $(\e_n)_{n\ge 1}$ be a sequence in $\R_{>0}$ that belongs to the set
$\ell_{\ge 0}^{1,1+2s}\setminus\cup_{\sigma>s}\ell_{\ge 0}^{1,1+2\sigma}$ 
and satisfies
\begin{equation}\label{bound epsilons}
4\sum_{n=1}^\infty \e_n<1\ .
\end{equation}
By the density of  $\omega \cdot \Z^2$ in $\R$,  there exist  $ m^{(n)} \in \Z^2$, $n \ge 1$, so that 
$$2+\e _n\leq  m^{(n)} \cdot \omega \leq 2+2\e_n\ .$$
For any $n \ge 1$, let $\overline \gamma_n$ be the number in $[\e_n/2, \e_n]$, defined by
$$ 
m^{(n)} \cdot \omega =2+2\overline \gamma_n \,  .
$$
By \eqref{bound epsilons} it then follows that 
$$
x :=1-2\sum_{j=1}^\infty \overline \gamma_j \in\Big(\frac{1}{2},1\Big) \,  .
$$
{\em Step 2.} We correct $m^{(n)}$ and $\overline \gamma_n$ so that \eqref{condition k^1} is satisfied
with $k^{(1)} = 0$.
To this end, we inductively construct a sequence $(\delta_n)_{n\ge 1}$ in $\R_{>0}$,
which belongs to $\omega \cdot \Z^2$ and satisfies
\begin{equation}\label{delta}
x=\sum_{j=1}^\infty \delta_j \, , \qquad   0 < \delta_n< 2^{1-n} \, ,  \quad   \forall \, n\ge 1\,  .
\end{equation}
We begin with $\delta_1$. Since $\frac 12 <x<1$,
there exists $\delta_1\in \omega \cdot  \Z^2 $ so that
$x -\frac 12 <\delta_1 <x - \frac 14$.
It follows that $0 < \delta_1< 1$ and that $y_1:= x - \delta_1$ satisfies
$$
\frac 14 < y_1 < \frac 12 \, , \qquad x = \delta_1 + y_1\, .
$$ 
Since $\frac 14 < y_1 < \frac 12$,
there exists $\delta_2\in \omega \cdot \Z^2 $ so that
$y_1-\frac 14<\delta_2<y_1-\frac 18$.
One concludes that $0<\delta_2<2^{-1}$ and that $y_2:= y_1 - \delta_2$ satisfies
$$\frac 18< y_2 < \frac 14\, , \qquad
x = \delta_1 + \delta_2 + y_2    \,  .$$
Continuing inductively in this way, we construct sequences $(y_n)_{n \ge 1}$, 
$(\delta_n)_{n \ge 1}$ in $\R_{>0}$ with $(\delta_n)_{n \ge 1}$ belonging to $\omega \cdot \Z^2$, so that for any $n \ge 1$,
$$
0 < \delta_n < \frac{1}{2^{n-1}} \, , \qquad
\frac 1{2^{n+1}}< y_n< \frac 1{2^n}  \, ,  \qquad    
x = \sum_{j = 1}^n \delta_j + y_n\, .
$$
Hence we obtain
$ x=\sum_{j=1}^\infty \delta_j$.
By construction,   $\delta_n$ is of the form $\delta_n =  p^{(n)} \cdot \omega$ with $p^{(n)}\in \Z^2$ 
and hence we define
$$
\gamma_n :=\overline \gamma_n +\frac{\delta_n}{2} > 0  \, ,  \quad 
\ell^{(n)}:=  m^{(n)} +p^{(n)} \in \Z^2  \,  , \qquad \forall \, n \ge 1 \, .
$$
Since $(\overline \gamma_n)_{n\ge 1}$ is in 
$\ell_{\ge 0}^{1,1+2s} \setminus 
\cup_{\sigma>s} \ell_{\ge 0}^{1,1+2\sigma}$,
and since  $\delta_n$ satisfies $0 < \delta _n < 2^{1-n}$ for any $n \ge 1,$
$(\gamma_n)_{n\ge 1}$ is also in $\ell_{\ge 0}^{1,1+2s}\setminus \cup_{\sigma>s}\ell_{\ge 0}^{1,1+2\sigma}$. Furthermore,
$$
2+2\gamma_n=  2 + 2 \overline \gamma_n +\delta_n 
= m^{(n)} \cdot \omega+p^{(n)}\cdot\omega =  \ell^{(n)} \cdot \omega  
$$
and $k^{(1)} = 0$, as
$
1-2\sum_{j=1}^\infty \gamma_j= 1 - 2\sum_{j=1}^\infty \overline \gamma_j  - \sum_{j=1}^\infty \delta_j=0=  0 \cdot \omega  \,  .
$
\end{proof}
\begin{remark}
In contrast to the periodic in time solutions of \eqref{BO} of Theorem \ref{periodic solutions which are not finite gap}
(cf. also Remark \ref{lacunary structure}(ii)),
the action variables of the $\omega$--quasiperiodic solutions constructed in the proof of Theorem \ref{nonsmooth qp solutions} 
are all strictly positive. 
\end{remark}


\appendix

\section{Generating function}\label{generating function}

The aim of this appendix is to show that the generating function $\mathcal H_\lambda(u)$,
defined in \eqref{def generating function} (cf. \cite{GK}, \cite{GKT1}), 
is the relative determinant of  $L_u + \lambda + 1$ by $L_u + \lambda$,
where $L_u$ is the Lax operator of \eqref{BO} (cf. \eqref{Lax}). 

First, let us introduce the notion of a relative determinant in a setup, sufficient for our purposes. 
To motivate the definition of such a determinant, consider a positive Hermitian $N \times N$ matrix $A$
with complex valued coefficients.
We list its eigenvalues in increasing order and with their multiplicities,  $\mu_1 \le \mu_2 \le  \cdots \le \mu_N$
and  consider the one parameter family of matrices $A+ \lambda$.  
It is then straightforward to verify that for any $\lambda > - \mu_1$,  the following, well-known formula holds,
$$
\frac{d}{d\lambda} 
\log (\mathrm{det}(A + \lambda) )
= \text{\rm trace} (A + \lambda)^{-1}\, .
$$
This formula motivates the following definition of a relative determinant.
Let $A$ and $A_0$ be two self-adjoint operators,  
acting on the Hardy space $H_+$ and assume that both are semibounded (from below) and 
have compact resolvents. Hence their spectrum is discrete, real, and bounded from below.
The spectrum of $A$ and the one of $A_0$ then consist of  increasing sequences $(\mu_n)_{n \ge 0}$ 
and  $(\nu_n)_{n \ge 0}$,  respectively, of real eigenvalues, converging to $\infty$.
Let $\lambda_* :=  \max\{- \mu_0, - \nu_0\}$. Then for any $\lambda > \lambda_*$, 
$A+\lambda$ and $A_0+\lambda$ are invertible.
We say that $A + \lambda$ admits a determinant relative to $A_0 + \lambda$ (for $\lambda > \lambda_*$), 
if for any  $\lambda > \lambda_*$, $(A + \lambda)^{-1} - (A_0 + \lambda)^{-1}$ is of trace class and 
if there exists a $C^1$-function, 
$$
( \lambda_*, \infty) \to \mathbb R_{>0}, \, \lambda \mapsto  \mathrm{det}_{A_0 + \lambda}(A + \lambda) \, ,
$$
satisfying the normalization condition $\det_{A_0 + \lambda}(A + \lambda) \to 0$ as $\lambda \to \infty$,
and for any $\lambda_* < \lambda < \infty $, the variational formula
$$
\frac{d}{d\lambda} 
\log \big(\mathrm{det}_{A_0 + \lambda}(A + \lambda) \big)
= \text{\rm trace} \big( (A +\lambda)^{-1} - (A_0 + \lambda)^{-1} \big) \, .
$$
To state our result on $\mathcal H_\lambda(u)$, recall that
for any $u \in H^s_{r, 0}$ with $s > -1/2$,  the spectrum of the Lax operator $L_u$
is given by a sequence of simple real  eigenvalues,  $\lambda_0 < \lambda_1 < \cdots$, and  (cf. \cite{GK}, \cite{GKT1})
\begin{equation}\label{formula eigenvalues}
 \lambda_n = n - \sum_{k > n} \gamma_k\, , \qquad   \gamma_k:= \lambda_k - \lambda_{k -1} - 1 \ge 0\, , \quad \forall \, k \ge 1\, .
 \end{equation}
The sequence $(\gamma_k)_{k \ge 1}$ is an element in the weighted $\ell^1$-space $\ell^{1, 1 + 2s}$. 
\begin{proposition}
For any $u \in H^{s}_{r,0}$ with $s > -1/2$ the following holds:\\
(i) For any $\lambda > - \lambda_0$,  $(L_u + 1 + \lambda)^{-1} - (L_u + \lambda)^{-1}$ is of trace class.\\
(ii) $L_u + 1+ \lambda$ admits a determinant relative to $L_u + \lambda$ and 
$$
\mathcal H_\lambda (u) = \mathrm{det}_{L_u + \lambda}(L_u + 1 +\lambda)\, , \qquad \forall \, \lambda >  - \lambda_0 \, .
$$
\end{proposition}
\begin{proof} We use arguments developed in the proof of \cite[Lemma 3.2]{GK}.\\
(i)  By functional calculus, $(L_u + 1 + \lambda)^{-1} - (L_u + \lambda)^{-1}$ is a self-adjoint operator on $H_+$ with eigenvalues
$$
\frac{1}{\lambda_n + 1 + \lambda } - \frac{1}{\lambda_n  + \lambda} = -  \frac{1}{(\lambda_n + 1 + \lambda )(\lambda_n  + \lambda )} \, .
$$
One then concludes from \eqref{formula eigenvalues} and the decay properties of the $\gamma_n$
that $(L_u + 1 + \lambda)^{-1} - (L_u + \lambda)^{-1}$ is of trace class for any $\lambda > - \lambda_0$ and that its trace is given by
$- \sum_{n \ge 0}  (\lambda_n + 1 + \lambda )^{-1}(\lambda_n  + \lambda )^{-1}$. \\
(ii) Denote by $S$ the shift operator on $H_+$ and by $S^*$ its adjoint,
$$
S : H_+ \to H_+, \, f \mapsto e^{ix} f\, \qquad  S^* : H_+ \to H_+, \, f \mapsto e^{-ix} (f - \langle f | 1 \rangle 1)\, .
$$
By a straightforward computation, one has (cf. \cite[(3.3)]{GK})
\begin{equation}\label{identity 1}
S^*(L_u + \lambda) S = L_u + \lambda + 1
\end{equation}
and (cf. \cite[Lemma 3.1]{GK})
\begin{equation}\label{identity 2}
(S^*(L_u + \lambda) S)^{-1} = S^* (L_u + \lambda)^{-1} S - 
\frac{\langle \cdot \, | \, S^*w_\lambda \rangle}{\langle w_\lambda \, | \, 1 \rangle} S^*w_\lambda\, ,
\end{equation}
where $w_\lambda := (L_u + \lambda)^{-1}1$.
Combining the identities \eqref{identity 1}-\eqref{identity 2}, one gets
\begin{equation}\label{trace 1}
(L_u + 1 + \lambda)^{-1} - (L_u + \lambda)^{-1} = I_\lambda -  II_\lambda \, ,
\end{equation}
where
$$
I_\lambda := S^* (L_u + \lambda)^{-1} S - (L_u + \lambda)^{-1} \, ,
\qquad
II_\lambda := \frac{\langle \cdot \, | \, S^*w_\lambda \rangle}{\langle w_\lambda \, | \, 1 \rangle} S^*w_\lambda\, .
$$
Note that $II_\lambda$ is an operator of rank 1 and hence in particular of trace class. Its trace can be computed as follows.
Denote by $(f_n)_{n \ge 0}$ the orthonormal basis of eigenfunctions of $L_u$, introduced in \cite{GK} ($s=0$) and \cite{GKT1} ($-1/2 < s < 0$).
Computing the trace of $II_\lambda$ with respect to this basis one obtains from Parseval's equality,
$$
\text{\rm trace} (II_\lambda) = \sum_{n \ge 0}  \frac{\langle f_n \, | \, S^*w_\lambda \rangle}{\langle w_\lambda \, | \, 1 \rangle} \langle S^*w_\lambda \, | \, f_n \rangle 
= \frac{\|S^*w_\lambda\|^2}{\langle w_\lambda \, | \, 1 \rangle}\,.
$$
Using that $S^* w_\lambda = e^{-ix} w_\lambda - \langle w_\lambda | 1 \rangle e^{-ix}$ and that 
$$
\langle w_\lambda | 1 \rangle = \langle (L_u + \lambda)^{-1} 1 | 1 \rangle \ge 0 
$$ 
since $(L_u + \lambda)^{-1}$ is a positive operator, one infers
\begin{equation}\label{trace II lambda}
\rm{trace} (II_\lambda) =  \frac{\| S^*w_\lambda \|^2}{\langle w_\lambda \, | \, 1 \rangle}
= \frac{\| w_\lambda \|^2}{\langle w_\lambda \, | \, 1 \rangle} -  \frac{| \langle w_\lambda | \, 1 \rangle |^2 }{\langle w_\lambda \, | \, 1 \rangle}
=  \frac{\| w_\lambda \|^2}{\langle w_\lambda \, | \, 1 \rangle} -  \langle w_\lambda | \, 1 \rangle \, .
\end{equation}
Since, by \eqref{trace 1}, the operator $I_\lambda$ is the sum of two operators of trace class, it is itself of trace class. 
Computing its trace with respect to the orthonormal basis $(e^{inx})_{n \ge 0}$, one obtains
\begin{align}\label{trace I lambda}
\text{\rm trace} (I_\lambda)  & = \sum_{n \ge 0} \langle (L_u +\lambda)^{-1} Se^{inx} \, | \, S e^{inx} \rangle  - \langle (L_u +\lambda)^{-1} e^{inx} \, | \, e^{inx} \rangle \nonumber\\
& = - \langle (L_u +\lambda)^{-1} 1  \, | \, 1 \rangle = - \langle w_\lambda | \, 1 \rangle\, .
\end{align}
Combining \eqref{trace 1}, \eqref{trace II lambda}, and \eqref{trace I lambda} we arrive at
$$
\text{\rm trace} \big( (L_u + 1 + \lambda)^{-1} - (L_u + \lambda)^{-1} \big) 
= -  \frac{\| w_\lambda \|^2}{\langle w_\lambda \, | \, 1 \rangle}\, .
$$
Since
$$
\langle w_\lambda \, | \, 1 \rangle = \sum_{n \ge 0} \frac{| \langle 1 \, | \, f_n \rangle |^2}{\lambda_n +\lambda}\, , \qquad
\| w_\lambda \|^2 = \sum_{n \ge 0} \frac{| \langle 1 \, | \, f_n \rangle |^2}{(\lambda_n +\lambda)^2}
$$
one has\footnote{Note that by \eqref{def generating function}, for a given $u\in H^s_{r,0}$,
${\mathcal H}_\lambda(u)$ is real analytic for 
$\lambda>-\lambda_0$.}
$$
-\frac{\| w_\lambda \|^2}{\langle w_\lambda \, | \, 1 \rangle} = 
\frac{d}{d\lambda} \log\Big( \sum_{n \ge 0} \frac{| \langle 1 \, | \, f_n \rangle |^2}{\lambda_n +\lambda}\Big)
$$
and hence
\begin{equation}\label{variational formula for mathcal H}
\frac{d}{d\lambda} \log\Big( \sum_{n \ge 0} \frac{| \langle 1 \, | \, f_n \rangle |^2}{\lambda_n +\lambda}\Big) = 
\text{\rm trace} \big( (L_u + 1 + \lambda)^{-1} - (L_u + \lambda)^{-1} \big) \, .
\end{equation}
Since when expanding $\mathcal H_\lambda(u) = \langle (L_u + \lambda \text{Id})^{-1} 1 | \, 1 \rangle$  
with respect to the orthonormal basis $(f_n)_{n \ge 0}$, one obtains
$$
\mathcal H_\lambda(u) = \sum_{n \ge 0}\frac{| \langle 1 \, | \, f_n \rangle |^2}{\lambda_n + \lambda}
$$
and since $\lim_{\lambda \to \infty} \mathcal H_\lambda(u) = 0$ we proved that $\mathcal H_\lambda(u)$
is the determinant of $L_u + 1 + \lambda$ relative to $L_u + \lambda$.
\end{proof}

\smallskip

In the remaining part of this appendix we study in more detail, how
for any given $u \in H^s_{r,0}$, $s > -1/2$, 
$\mathcal H_\lambda(u)$ is related to the spectrum of $L_u$.
First note that it follows from  \eqref{variational formula for mathcal H} that 
$$
\frac{d}{d\lambda} \log \mathcal H_\lambda  =  \frac{d}{d\lambda} \left( \log\Big(\frac{1}{\lambda_0 + \lambda}\Big) 
+\sum_{n \ge 1} \log\Big( \frac{\lambda_{n-1} + 1 + \lambda}{\lambda_{n}  + \lambda}\Big)\right)\, ,
$$
implying that (cf. \cite[Proposition 3.1]{GK}, \cite{GKT1})
\begin{equation}\label{product representation}
\mathcal H_\lambda(u) = \frac{1}{\lambda_0 + \lambda} \prod_{n\ge 1} \frac{\lambda_{n-1}  + 1 + \lambda}{\lambda_n + \lambda}\ . 
\end{equation}
Therefore, $\mathcal H_\lambda(u)$ is determined by the periodic spectrum of $L_u$. Since the latter is invariant by the flow of \eqref{BO},
$\mathcal H_\lambda(u)$ is a one parameter family of prime integrals of this equation. 
The question arises if conversely,  $\mathcal H_\lambda(u)$ determines the spectrum of $L_u$.
To this end we take a closer look at the product representation \eqref{product representation} of $\mathcal H_\lambda(u)$.
Setting $ \nu_n := \lambda_{n-1} +1$ for any $n \ge 1$, one has
$$
\lambda_0 < \nu_1 \le \lambda_1 < \nu_2 \le \lambda_2 < \cdots \, .
$$
Furthermore, for any $n \ge 1$, $ \nu_n = \lambda_n$ if and only if $\gamma_n = 0$. Hence
$$
\mathcal H_\lambda(u) = \frac{1}{\lambda_0 + \lambda} \prod_{n\in J_u} \frac{\nu_n + \lambda}{\lambda_n + \lambda} \, ,
\qquad  J_u:= \{ n \ge 1 \,|\, \gamma_n > 0  \} \, .
$$
Note that $\nu_n$, $n \in J_u$, are the zeros of $\mathcal H_\lambda(u) $ and $\lambda_n$, $n \in J_u\cup \{0\}$, its poles.
All the poles and zeros of $\mathcal H_\lambda(u)$ are simple.
The question raised above can now be rephrased as follows: 
does $\mathcal H_\lambda(u)$ besides $\nu_n(u)$, $n \in J_u$, and $\lambda_n(u)$, $n \in J_u\cup\{0\}$,
also determine the eigenvalues $\lambda_n(u)$ with $\gamma_n(u) = 0$?
The following result says that this is indeed the case.

\begin{proposition}\label{eigenvalues with gamma 0}
For any $u \in H^s_{r, 0}$, $s > -1/2$, the generating function $\mathcal H_\lambda(u)$ 
determines the entire spectrum of $L_u$.
\end{proposition}

\begin{proof} Since $\lambda_0$ is a pole of $\mathcal H_\lambda(u)$, it is determined by the generating  function.
If $\nu_1\equiv\lambda_0 + 1$ is a zero  of $\mathcal H_\lambda(u)$, then $\nu_1 = \lambda_0 + 1 < \lambda_1$ and hence 
$\lambda_1$ is a pole of $\mathcal H_\lambda(u)$. If $\lambda_0 + 1$ is not a zero of $\mathcal H_\lambda(u)$, then 
$ \lambda_0 + 1$ is the periodic eigenvalue $\lambda_1$ of $L_u$ and hence also determined by $\mathcal H_\lambda(u)$. 
Arguing inductively, the claimed result follows.
\end{proof}

\begin{remark} 
With the help of a conformal map, Hochstadt proved
a result corresponding to the one of Proposition \ref{eigenvalues with gamma 0} for Hill's operator, which is a
Lax operator for the Korteweg-de Vries equation (cf. \cite{Hoch},  \cite{McKT}). Note that in contrast, the proof 
of Proposition \ref{eigenvalues with gamma 0} is elementary.
\end{remark}



\begin{thebibliography}{99}



\bibitem{ABFS} {\sc L. Abdelouhab, J. Bona, M. Felland,
J.-C. Saut}, {\em Non local models for nonlinear dispersive waves}, Physica D, Nonlinear Phenomena, 40(1989), 360--392

\bibitem{AT} {\sc C. Amick, J. Toland}, 
{\em Uniqueness and related analytic properties for the Benjamin-Ono equation -- a nonlinear Neumann problem in the plane}, 
Acta Math., 167(1991), 107--126

\bibitem{AN} {\sc J.  Angulo Pava, F.  Natali}, 
{\em Positivity properties of the Fourier transform and the stability of periodic travelling-wave solutions}, 
SIAM J. Math. Anal. 40 (2008), no. 3, 1123--1151

\bibitem{AH} {\sc J.  Angulo Pava, S. Hakkaev}, {\em Illposedness for periodic nonlinear dispersive equations}, 
Elec. J. Differential Equations, vol.  2010 (2010), no. 119, 1-19

\bibitem{Benj} {\sc T. Benjamin}, 
{\em Internal waves of permanent form in fluids of great depth,}
J. Fluid Mech., 29(1967), 559--592

\bibitem{BK} {\sc T. Bock, M. Kruskal}, {\em A two-parameter Miura transformation of the Benjamin-Ono equation,} 
Phys. Lett. A, 74(1979), 173--176

\bibitem{Deng} {\sc  Y. Deng}, {\em Invariance of the Gibbs measure for the Benjamin-Ono equation}.
 J. Eur. Math. Soc. (JEMS) 17(2015), no. 5, 1107--1198

\bibitem{DTV} {\sc  Yu. Deng, N. Tzvetkov, N. Visciglia}, {\em Invariant measures and long time behaviour for the Benjamin-Ono equation III},
 Comm. Math. Phys. 339(2015), no. 3, 81--857
 
\bibitem{DK} {\sc S. Dobrokhotov, I. Krichever}, {\em Multi-phase solutions of the Benjamin-Ono equation and their averaging}, 
Math. Notes 49(1991), no. 5--6, 583--594

\bibitem{GK} {\sc P. G\'erard, T. Kappeler},
{\em On the integrability of the Benjamin--Ono equation on the torus}, arXiv:1905.01849, to appear in Comm. Pure and Appl. Math.

\bibitem{GKT1} {\sc P. G\'erard, T. Kappeler, P. Topalov}, {\em Sharp well-posedness results of the Benjamin-Ono equation in
$H^{s}(\T,\R)$ and qualitative properties of its solutions},  arXiv:2004.04857

\bibitem{GKT2} {\sc P. G\'erard, T. Kappeler, P. Topalov}, {\em On the spectrum of the Lax operator of the 
Benjamin-Ono equation on the torus}, to appear in J. of Funct. Analysis, arXiv:2006.11864

\bibitem{GKT3} {\sc P. G\'erard, T. Kappeler, P. Topalov}, {\em On the analytic Birkhoff normal form of the Benjamin-Ono equation and applications}, preprint

\bibitem{GKT4} {\sc P. G\'erard, T. Kappeler, P. Topalov}, {\em On the analyticity of the Birkhoff map of the Benjamin-Ono 
equation in the large}, in preparation


\bibitem{Hoch} {\sc H. Hochstadt}, {\em Function-theoretic properties of the discriminant of Hill's equation}, 
Math. Zeitschrift 83(1963), 237--242

\bibitem{KMMT} {\sc  T. Kappeler, A. Maspero, J. Molnar, P. Topalov}, {\em On the convexity of the KdV Hamiltonian}, 
Comm. Math. Phys. 346(2016), no. 1, 191--236

\bibitem{McKT} {\sc  H. McKean, E. Trubowitz}, {\em Hill's operator and hyperelliptic function theory in the presence of infinitely many branch points},
 Comm. Pure Appl. Math. 29(1976), no. 2, 143--226
 
 \bibitem{Mol} {\sc L. Molinet}, 
{\em  Global well-posedness in $L^2$ for the periodic Benjamin-Ono equation},
American J. Math.130 (3) (2008), 2793--2798

\bibitem{Mol0} {\sc L. Molinet}, {\em private communication}
 
 \bibitem{MP} {\sc L. Molinet, D. Pilod}, 
{\em The Cauchy problem for the Benjamin-Ono equation in $L^2$ revisited},
 Anal. and PDE 5(2012), no. 2, 365--395
 
 \bibitem{Molnar} {\sc J. Molnar}, {\em Features of the Nonlinear Fourier Transform for the dNLS Equation}, 
 Ph.D. thesis, University of Zurich, Zurich, 2016
 
 \bibitem{Nak} {\sc A. Nakamura}, {\em Backlund transform and conservation laws of the Benjamin-Ono equation},
J. Phys. Soc. Japan 47(1979), 1335--1340 


\bibitem{SI} {\sc J. Satsuma, Y. Ishimori}, {\em Periodic wave and rational soliton solutions of the Benjamin-Ono equation}, J. Phys. Soc. Japan 46(1979), 681--687

\bibitem{Sa0} {\sc J.-C. Saut}, {\em Sur quelques 
g\'en\'eralisations de l'\'equation de Korteweg-de Vries},
J. Math. Pures Appl. 58(1979), 21--61

\bibitem{Sa} {\sc J.-C. Saut}, 
{\em Benjamin-Ono and Intermediate Long Wave equations: modeling, IST, and PDE},  
in {\em Nonlinear partial differential equations and inverse scattering}, 95--160.
Fields Institute Communications 83, Miller, Perry, Saut, Sulem eds, Springer, New York, 2019

\bibitem{Tal} {\sc B. Talbut}, 
{\em Low regularity conservation laws for the Benjamin-Ono equation}, arXiv:1812.00505,
to appear in Math. Research Letters

\bibitem{W} {\sc H. Weyl}, {\em \"Uber die Gleichverteilung von Zahlen mod. Eins}, Math. Ann. 77(1916), 313--352

\end{thebibliography}
\end{document}